\documentclass[12pt,twoside]{article} 
\usepackage{amsmath, amssymb}
\usepackage{amsthm} 
\theoremstyle{plain} 
\newtheorem{theorem}{\indent\bf Theorem}[section]

\newtheorem{corollary}[theorem]{\indent\bf Corollary}

\theoremstyle{definition} 

\makeatletter
\def\address#1#2{\begingroup
\noindent\parbox[t]{7.8cm}{%
\small{\scshape\ignorespaces#1}\par\vskip1ex
\noindent\small{\itshape E-mail address}%
\/: #2\par\vskip4ex}\hfill%
\endgroup}%
\makeatother
\title{ {\bf Surfaces of Revolution with Constant Gaussian Curvature in Four-Space}} 
\author{ Dang Van Cuong
\bigskip
\textsc{\ }}
\date{\today}
\begin{document}
\maketitle

\begin{abstract}
In this paper, we  show that the constant property of the Gaussian curvature of surfaces of revolution in both $\mathbb R^4$ and $\mathbb R_1^4$ depend only on the radius of rotation. We then give necessary and sufficient conditions for the Gaussian curvature of the general rotational surfaces whose meridians lie in two dimensional planes in $\mathbb R^4$ to be constant, and define the parametrization of the meridians when both the Gaussian curvature is constant and the rates of rotation are equal.
\end{abstract}

\section{Introduction}
It is well known that a regular surface in $\mathbb R^3$ is zero Gaussian curvature if and only if it is a part of a developable surface. A regular surface with constant Gaussian curvature, $K,$ is locally isometric to $H^2$ provided $K=-1,$ $\mathbb R^2$ provided $K=0,$ and $S^2$ provided $K=1.$ For the surfaces of revolution in $\mathbb R^3,$ it is easy to define the parametrization of the surfaces with constant Gaussian curvature.\\
\indent In recent years some mathematicians have taken an  interest in the surfaces of revolution in $\mathbb R^4,$  for example V. Milosheva (\cite{sheva}), U. Dursun and N. C. Turgay (\cite{dursun1}), K. Arslan (\cite{arslan}), \dots. In \cite{Ganchev}, V. Milosheva applied  invariance theory of surfaces in the four dimensional Euclidean space to the class of general rotational surfaces whose meridians lie in two-dimensional planes in order to find all minimal super-conformal surfaces. These surfaces were further studied by U. Dursun and N. C. Turgay in \cite{dursun1}, which found all minimal surfaces by solving the differential equation that characterizes minimal surfaces. They then determined all pseudo-umbilical general rotational surfaces in $\mathbb R^4$.  K. Arslan et.al in \cite{arslan}  gave the necessary and sufficient conditions for generalized rotation surfaces to become pseudo-umbilical, they also shown that each general rotational surface is a Chen surface in $\mathbb E^4$ and  gave some  special classes  of generalized rotational surfaces as examples.\\
\indent Let $M$ be a spacelike or timelike surface in Lorentz-Minkowski three-space $\mathbb R_1^3$ generated by a one-parameter family of circular arcs, R. L${\rm \acute{o}}$pez in \cite{lopez} shown that if its Gaussian curvature $K$ is a nonzero constant then $M$ is a surface of revolution, he also described the parametrizations for $M$ when $K=0.$ In \cite{C}, by  applying the $\mathfrak l_r^{\pm} $-Gauss maps, Cuong defined the parameterizations of minimal and totally umbilical  spacelike surfaces of revolution in $\mathbb R_1^4$.\\
\indent In this paper, we introduce the notions of surfaces of revolution in $\mathbb R^4$ and in Lorentz-Minkowski $\mathbb R_1^4,$ we then give necessary and sufficient conditions for the Gaussian curvature of these surfaces to be constant. We then give a differential  equation that characterizes the general rotational surfaces whose meridians lie in two dimensional planes in $\mathbb R^4$  with constant Gaussian curvature. In the case that rates of rotation are equal, we can define the parametrization of the meridians of these surfaces when its  Gaussian curvature is constant.
\section{Preliminaries}
Let $M$ be a semi-Riemannian surface, that is, a semi-Riemannian manifold of dimension two. For a coordinate system $u,v$ in $M$ the components of the metric tensor (the coefficients of the first fundamental form) are traditionally denoted by
$$E=g_{11}=\langle \partial_u,\partial_u\rangle, F=g_{12}=\langle \partial_u,\partial_v\rangle, G=g_{22}=\langle \partial_v,\partial_v\rangle.$$
Since $M$ is two-dimensional, $T_pM$ is the only tangent plane at $p$. Thus the sectional curvature $K$ becomes a real-valued function on $M$, called Gaussian curvature of $M.$\\
\indent Let $u,v$ be an orthogonal coordinate system in a semi-Riemannian surface, that means $F=\langle \partial_u,\partial_v\rangle=0.$ Then (see Proposition 4.4, pp. 81, \cite{onel})
$$K=-\frac{1}{eg}\left[\varepsilon_1\left(\frac{g_u}{e} \right)_u+\varepsilon_2\left(\frac{e_v}{g} \right)_v  \right],  $$
where $e=|E|^{1/2},g=|G|^{1/2}$ and $\varepsilon_1,\varepsilon_2$ are the sign of $E,G$ , respectively.\\
\indent If $M$ is a surface immersed in a manifold of constant curvature $C,$ the normal bundle of $M$ has a orthonormal frame $\{\nu_i\}_{i=1,\dots,n}$ then by using the equation of Gauss we have
$$K=C+\sum_{i=1}^nK_{\nu_i},$$
where $K_{\nu_i}$ is $\nu_i$-curvature of $M$ associated with $\nu_i.$ For more detail, let see \cite{M.Navarro}.\\
\indent Therefore, if $M$ is a surface immersed in a manifold with constant curvature  we then can define the Gaussian curvature by two ways. In this paper, we only use the formula of Gaussian curvature in term of the coefficients of the first fundamental form of $M,$ and apply this formula to define the surfaces of revolution whose Gaussian curvature are constant. \\
\indent We now introduce the notion surfaces of revolution in $\mathbb R^4.$ Let $C$ be a curve in ${\rm span}\{e_1,e_2,e_3\}$ parametrized by arc-length
\begin{equation}\label{C1} z(u)=\left(f(u),g(u),\rho(u),0\right),\ u\in I, \end{equation}
where $\rho(u)>0$.
The orbit of $C$ under the action of the orthogonal transformations of $\mathbb R^4$  leaving the  plane $Oxy,$
$$A=\left[\begin{matrix}1&0&0&0\\0&1&0&0\\0&0&\cos v&-\sin v\\0&0&\sin v&\cos v\end{matrix}\right],\ v\in\mathbb R,$$
is a  surface given by
\begin{equation}\label{R1}{[SR1]} \qquad {{\rm X}}(u,v)=\left(f(u),g(u),\rho(u)\cos v,\rho(u)\sin v\right), u\in I, v\in[0,2\pi). \end{equation}
Then $[SR_1]$ is called {\it surface of revolution} in $\mathbb R^4.$ That means, $[RS_1]$ is orbit of a curve by rotating it around a plane. \\
We also have the another kind of surface of revolution in $\mathbb R^4,$ it is the obit of a plane curve rotated  around both two planes. This surface is defined as following. Let $C$ be a regular curve in $\text{span}\{e_1,e_3\}$ parametrized by arc-length
$$r(u)=\left(f(u),0,g(u),0\right),\ u\in I,$$
and
$$B=\left[\begin{matrix}\cos\alpha v&-\sin\alpha v&0&0\\\sin\alpha v&\cos\alpha v&0&0\\0&0&\cos\beta v&-\sin\beta v\\0&0&\sin\beta v&\cos\beta v\end{matrix}\right],\ v\in\mathbb R,$$
be a subgroup of the  orthogonal transformations group  on $\mathbb R^4$, where  $\alpha,\beta$ are positive constants and $(f(u))^2+(g(u))^2\ne0$.\\
\indent The orbit of $C$ under the action of the subgroup $B$
is a  surface in $\mathbb R^4$ given by
\begin{equation}\label{gero}[SR_2]\qquad {\rm X}(u,v)=\left(f(u)\cos\alpha v,f(u)\sin\alpha v,g(u)\cos\beta v,g(u)\sin\beta v\right),\end{equation}
which is called {\it General rotational surface whose meridians lie in two-dimensional planes}. Then $r(u)$ is called meridian and $\alpha,\beta$ are called the rates of rotation.

Modifying this method, we can introduce the notion surfaces of revolution in Lorentz-Minkowski space. The Lorentz-Minkowski space $\Bbb R^{4}_1$ is the $4$-dimensional vector space $\Bbb R^{4}=\{( x_1,\ldots, x_{4}): x_i\in \Bbb R, i=1,\ldots, 4\}$ endowed the pseudo scalar product  defined by
$$\langle \textbf x, \textbf y\rangle=\sum_{i=1}^{3}x_iy_i-x_{4}y_{4},$$
where $\textbf x=( x_1,\ldots, x_{4}), \textbf y=(y_1, \ldots y_{4})\in \Bbb R^{4}.$\\
Let $C$ be a spacelike (timelike) curve in $\text{span}\{e_1,e_2,e_4\}$ parametrized by  arc-length,
$$z(u)=\left(f(u),g(u),0,\rho(u)\right),\ \rho(u)>0,\ \ u\in I.$$
The orbit of $C$ under the action of the orthogonal transformations of $\mathbb R_1^4$ leaving the spacelike plane $Oxy,$
$$A_S=\left[\begin{matrix}1&0&0&0\\0&1&0&0\\0&0&\cosh v&\sinh v\\0&0&\sinh v&\cosh v\end{matrix}\right],\ v\in\mathbb R,$$
is a  surface given by
\begin{equation}\label{hr1} [SR_3]\qquad
{\rm X} (u,v)=\left(f(u),g(u),\rho(u)\sinh v, \rho(u)\cosh v\right),\ u\in I,\ v\in\mathbb R.
\end{equation}

The surface $[SR_3]$ is called  the {\it surface of revolution of hyperbolic type} in $\mathbb R_1^4$. \\
\indent Let $C$ be a spacelike (timelike) curve in $\text{span}\{e_1,e_3,e_4\}$ parametrized by  arc-length,
$$z(u)=\left(\rho(u),0,f(u),g(u)\right),\ \rho(u)>0,\ u\in I.$$
The orbit of $C$ under the action of the orthogonal transformations of $\mathbb R_1^4$  leaving the timelike plane $Ozt,$
$$A_T=\left[\begin{matrix}\cos v&-\sin v&0&0\\\sin v&\cos v&0&0\\0&0&1&0\\0&0&0&1\end{matrix}\right],\ v\in\mathbb R,$$
is a  surface  given by
\begin{equation}\label{R2}[SR_4]\qquad {\rm X}(u,v)=\left(\rho(u)\cos v,\rho(u)\sin v,f(u),g(u)\right),\ v\in\mathbb R. \end{equation}
 The surface $[SR_4]$ then is called the {\it surface of revolution of elliptic type} in $\mathbb R_1^4.$
\section{Main Results}
The following theorems show that the constant property of Gaussian curvatures of surfaces of revolution in four-space depends only on the radius of rotation.
\begin{theorem}\label{thSR1} The Gaussian curvature of surface $[SR_1]$ is constant if and only if
\begin{enumerate}
\item $\rho(u)=C_1e^{Cu}+C_2e^{-Cu },$ provided $K=-C^2,$
\item $\rho(u)=C_1\sin(Cu)+C_2\cos(Cu),$ provided $K=C^2,$
\item $\rho=C_1u+C_2,$ provided $K=0,$
\end{enumerate}
where $C_1,C_2$ are constant such that for each $u\in I,$ $\rho(u)>0.$
 \end{theorem}
 \begin{proof}
 The coefficients of the first fundamental form of $[SR_1]$ are defined
 $$E=\langle {\rm X}_u,{\rm X}_u  \rangle=1,\ F=\langle {\rm X}_u,{\rm X}_v   \rangle=0,\ G=\langle {\rm X}_v,{\rm X}_v  \rangle=\rho^2.$$
 Therefore, the Gaussian curvature is defined
 $$K=-\frac{\rho''}{\rho}.$$
 Solving the equation $K=-C^2,$ ($C^2,0$) we have the conclusion of Theorem \ref{thSR1}.
 \end{proof}
 We have the similar result for surfaces of revolution in $\mathbb R_1^4.$
\begin{theorem}\label{thrSr3} The Gaussian curvature of Surface $[SR_3]$ (or $[SR_4]$) is constant, $K=C,$ if and only if
\begin{enumerate}
\item $\rho(u)=C_1e^{\lambda u }+C_2e^{-\lambda u },$ when $\varepsilon C=-2\\lambda^2<0,$
\item $\rho(u)=C_1\sin(\lambda u)+C_2\cos(\lambda u),$ when $\varepsilon C=\lambda ^2>0,$
\item $\rho=C_1u+C_2,$ when $C=0,$
\end{enumerate}
where $C_1,C_2$ are constant such that $\rho(u)>0,u\in I$ and $\varepsilon$ is the sign of $E.$
 \end{theorem}
 \begin{proof} The coefficients of the first fundamental form of $[SR_3]$ (or $[SR_4]$) are defined
$$E=(f'(u))^2+(g'(u))^2-(\rho'(u))^2=\varepsilon,\  F=0,\ G=\left(\rho(u)\right)^2>0,$$
where $\varepsilon=\pm1.$ It is similar to the proof of Theorem \ref{thSR1}, we have the result of this Theorem.\end{proof}
For the general rotational surfaces whose meridians lie in two-dimensional planes in $\mathbb R^4,$ the following Theorem gives us a differential equation that characterizes  the constant Gaussian curvature  surfaces.
\begin{theorem}\label{thrSR2}
The Gaussian curvature of $[SR_2]$ is constant if and only if
$$f(u)=\frac{\sqrt G}{\alpha}\cos\phi(u),\ g(u)=\frac{\sqrt G}{\beta}\sin\phi(u),$$
where $\phi(u)$ are solutions of the following equation
\begin{equation}\label{slsphi}
G\left(\frac{\sin^2\phi}{\alpha^2}+\frac{\cos^2\phi}{\beta^2} \right)(\phi')^2+2G\sin \phi\cos \phi\left(\frac{1}{\beta^2}-\frac{1}{\alpha^2}\right)\phi' +\frac{(G')^2}{4G}\left(\frac{\cos^2\phi}{\alpha^2}+\frac{\sin^2\phi}{\beta^2}\right) -1=0,
\end{equation}
and \begin{equation}\label{g21}G=\left(C_1e^{Cu}+C_2e^{-Cu}\right)^2,\  \text{if}\  K=-C^2,\end{equation}
\begin{equation}\label{g22}G=\left(C_1\sin(Cu)+C_2\cos(Cu)\right)^2,\  \text{if}\  K=C^2,\end{equation}
\begin{equation}\label{g23}G=\left(C_1u+C_2\right)^2,\  \text{if}\  K=0,\end{equation}
where $C,C_1,C_2$ are constant such that $G\ne 0$.
\end{theorem}
\begin{proof}
The coefficients of the first fundamental form of $[SR_2]$ are defined
\begin{equation}\label{hscb1SR2} E=(f')^2+(g')^2=1,\ F=0,\ G=\alpha^2f^2+\beta^2g^2.\end{equation}
We have
$$K=-\frac{\left(\sqrt{G}\right)_{uu}}{\sqrt G}.$$
Solving the equation $K=-C^2 (C^2,0),$ we have $G.$ Setting
$$\alpha f(u)=\sqrt G\cos(\phi(u)), \beta g(u)=\sqrt G\sin(\phi(u))$$
and substituting for (\ref{hscb1SR2}) we then have  (\ref{slsphi}).
\end{proof}
 In the case two rates of rotation are equal, $\alpha=\beta,$  we can solve the equation (\ref{slsphi}) and find the parametrization of the meridians.
\begin{corollary} If $\alpha=\beta$ then the Gaussian curvature of  $[SR_2]$ is constant if and only if
\begin{enumerate}
\item In the case $K=C^2,$
$$f(u)=\frac{C_1\sin(Cu)+C_2\cos(Cu)}{\alpha}\cos\phi(u),g(u)=\frac{C_1\sin(Cu)+C_2\cos(Cu)}{\alpha}\sin\phi(u)$$
where
$$\phi(u)=\int\frac{\sqrt{1-C^2[C_1\cos(Cu)-C_2\sin(Cu)]^2}}{C_1\sin(Cu)+C_2\cos(Cu)}du,$$
 $C_1,C_2$ are constants such that the formula under the integral sign is defined.
\item In the case $K=-C^2,$
$$f(u)=\frac{C_1e^{Cu}+C_2e^{-Cu}}{\alpha}\cos\phi(u),\ g(u)=\frac{C_1e^{Cu}+C_2e^{-Cu}}{\alpha}\sin\phi(u),$$
where
$$\phi(u)=\int\frac{\sqrt{1-C^2[C_1e^{Cu}-C_2e^{-Cu}]^2}}{C_1e^{Cu}+C_2e^{-Cu}}du,$$
 $C_1,C_2$ are constants such that the formula under the integral sign is defined.
\end{enumerate}
\begin{enumerate}
\item[3.] In the case $K=0,$
$$f(u)=\frac{C_1u+C_2}{\alpha}\cos\phi(u),g(u)=\frac{C_1u+C_2}{\alpha}\sin\phi(u),$$
where
$$\phi(u)=\int{\frac{\sqrt{1-C_1^2}}{C_1u+C_2}}du,$$
$C_1,C_2$ are constants such that $|C_1|\leq 1$ and $C_1u+C_2\ne0,\forall u\in I.$
\end{enumerate}
\end{corollary}
{\bf Acknowledgements} The author would like to thank Prof. Frank Morgan for his comment and  appended.

\address{ Dang Van Cuong\\
Department of Natural Sciences \\
Duy Tan University \\
Danang \\
Vietnam
}
{dvcuong@duytan.edu.vn}
\end{document}